\documentclass[12pt]{amsart}

\usepackage{enumerate}
\usepackage{graphicx,graphics}
\usepackage{amsfonts}
\usepackage{amssymb}
\usepackage{amsthm}
\usepackage{amsmath}
\usepackage{mathrsfs}
\usepackage{hyperref}
\input{xy}
\xyoption{all}

\newtheorem{theorem}{Theorem}[section]
\newtheorem{lemma}[theorem]{Lemma}

\DeclareMathOperator{\Aff}{Aff}

\DeclareMathOperator{\Int}{int}
\numberwithin{equation}{section}

\begin{document}

\title{A short proof of  Gr\"unbaum's Conjecture about affine invariant points}

\author{Natalia Jonard-P\'erez}


\address{Departamento de Matem\'aticas, Facultad de Ciencias,  Universidad Nacional Aut\'onoma de M\'exico. Circuito Exterior S/N, Cd. Universitaria, Colonia Copilco el Bajo, Delegaci\'on Coyoac\'an,  04510,  M\'exico D.F., M\'EXICO.}
\email{nat@ciencias.unam.mx}

\thanks{The author has been partially supported by PAPIIT grant IA104816  (UNAM, M\'exico).
}

\maketitle\markboth{ NATALIA. JONARD-P\'EREZ }
{A short proof of Gr\"unbaum's Conjecture}

\begin{abstract}
Let us denote by $\mathcal K_n$ the hyperspace of all convex bodies of $\mathbb R^n$ equipped with the Hausdorff distance topology. An affine invariant point $p$ is a continuous and $\Aff(n)$-equivariant map $p:\mathcal K_n\to \mathbb R^n$, where $\Aff(n)$ denotes the group of all nonsingular affine maps of $\mathbb R^n$.
For every $K\in\mathcal K_n$, let $\mathfrak{P}_n(K)=\{p(K)\in\mathbb R^n\mid p\text{ is an affine invariant point}\}$ and $\mathfrak{F}_n(K)=\{x\in\mathbb R^n\mid gx=x\text{ for every }g\in\Aff(n)\text{ such that }gK=K\}$.
In 1963, B. Gr\"unbaum conjectured that $\mathfrak{P}_n(K)=\mathfrak{F}_n(K)$ \cite{Grunbaum}. After some partial results, the conjecture was recently proven in \cite{Olaf}. 

In this short note we give a rather different, simpler and shorter proof of this conjecture, based merely on the topology of the action of $\Aff(n)$ on $\mathcal K_n$.
\end{abstract}

\section{Introduction}
By a convex body $K\subset\mathbb R^n$ we mean a compact convex subset of $\mathbb R^n$ with a non empty interior. We denote by $\mathcal K_n$ the set of all convex bodies of $\mathbb R^n$ equipped with the Hausdorff distance
$$d_H(A,B)=\max\left\{ \sup\limits_{b\in B}d(b,A),~~~~~\sup\limits_{a\in A} d(a, B)\right\},$$
where $d$ is the standard Euclidean metric on $\Bbb R^n$. For any  $x\in \mathbb R^n$,  $K\in \mathcal K_n$ and $\varepsilon>0$ we will use the following notations:
$$B(x,\varepsilon)=\{y\in\mathbb R^n\mid d(x,y)<\varepsilon\},$$
$$O(K, \varepsilon)=\{A\in \mathcal{K}_n\mid d_H(K,A)<\varepsilon\}.$$

Let us denote by $\Aff(n)$ the group of all nonsingular affine maps of $\mathbb R^n$. Namely, $g\in \Aff(n)$ if and only if there exist an invertible linear map $\sigma:\mathbb R^n\to \mathbb R^n$ and a fixed point $v\in\mathbb R^n$ such that 
$$gx=v+\sigma x,\quad \text{for every }x\in\mathbb R^n.$$
Observe that each element $g\in\Aff(n)$ satisfies that
\begin{equation*}
g(tx+(1-t)y)=tgx+(1-t)gy,\quad \text{for every }x,y\in\mathbb R^n, \text{ and }t\in[0,1].
\end{equation*}

The natural action of $\Aff(n)$ on $\mathbb R^n$ induces a continuous action of $\Aff(n)$ on $\mathcal K_n$ by means of the formula
\begin{equation*}
(g, K)\longmapsto gK, \quad
gK=\{gx\mid x\in K\}.
\end{equation*}
We direct the reader to  \cite{NataliaFM} for more details about this action. 

In his famous paper \cite{Grunbaum}, B. Gr\"umbaum introduced the notion of an affine invariant point. Namely, an \textit{affine invariant point} is a continuous map $p:\mathcal K_n\to \mathbb R^n$ such that
$$gp(K)=pg(K),\text{ for every }g\in\Aff(n),\text{ and }K\in\mathcal K_n.$$
The centroid, the center of John's ellipsoid or the center of L\"owner's ellipsoid are examples of affine invariant points. 
If $p:\mathcal K_n\to \mathbb R^n$ is an affine invariant point and satisfies
$$p(A)\in \Int(A)\quad \text{ for every }A\in\mathcal K_n$$
then we call $p$ a \textit{proper affine invariant point}.
Using standard notation, let us denote by $\mathfrak{P}_n$ the set of all affine invariant points of $\mathcal K_n$. 
For every $K\in\mathcal K_n$, let $\mathfrak{P}_n(K)$ be the set of all $x\in\mathbb R^n$ such that $x=p(K)$ for some $p\in\mathfrak{P}_n$ and
$$\mathfrak{F}_n(K)=\{x\in\mathbb R^n\mid gx=x \text{ for every }g\in\Aff(n)\text{ such that }gK=K\}.$$
Observe that if  $x\in\mathfrak{P}_n(K)$, then there exists $p\in \mathfrak P_n$ with  $p(K)=x$. Therefore, if $g\in \Aff(n)$ satisfies $gK=K$ then we have that
$$x=p(K)=p(gK)=gp(K)=gx$$
and thus $x\in\mathfrak{F}_n(K)$. This implies that $\mathfrak{P}_n(K)\subset\mathfrak{F}_n(K)$. In \cite{Grunbaum} Gr\"unbaum conjectured that $\mathfrak{P}_n(K)=\mathfrak{F}_n(K)$, and in the following 50 years several partial results were obtained (see, e.g., \cite{Kuchment, Meyer}). Recently a full proof of this conjecture was finally given by O. Mordhorst in \cite{Olaf}.

The aim of this work is to provide an alternative proof of Gr\"umbaum's conjecture which is simpler, shorter and  different from the one given in \cite{Olaf}.
Our proof is based merely on some well-known topological results. What is interesting about this way of approaching Gr\"umbaum's conjecture (which was open for more than 50 years) is that once we translate the problem to the language of topolo\-gical transformation group's theory, the proof is quite natural and simple.

\section{Proper actions}

In order to make clear our proof, let us recall some basic notions about the theory of $G$-spaces. We refer the reader to the monographs \cite{Bredon} and \cite{Palais2} for a deeper understanding of this theory.

If $G$ is a topological group and $X$ is a $G$-space, for any $x\in X$ we denote by $G_x$ the \textit{stabilizer}  of $x$, i.e., $G_x=\{g\in G \mid gx=x\}$. For a subset $S\subset X$, the symbol $G(S)$ denotes the $G$-\textit{saturation} of $S$, i.e., $G(S)=\{gs\mid g\in G,\; s\in S\}.$ If $G(S)=S$ then we say that $S$ is a $G$-\textit{invariant} set (or, simply, an \textit{invariant} set). In particular, $G(x)$ denotes the $G$-\textit{orbit} of $x$, i.e., $G(x)=\{gx\in X\mid g\in G\}$. The set consisting of all orbits of $X$ equipped with the quotient topology is  denoted by $X/G$ and it is called the \textit{orbit space}. The quotient map $\pi:X\to X/G$ is called the \textit{orbit map} and it is always open.

For each subgroup $H\subset G$, the $H$-\textit{fixed point set} $X^H$ is the set $\{x\in X\mid H\subset G_x\}$. Clearly, $X^{H}$ is a closed subset of $X$.

A continuous map $f:X\to Y$ between two $G$-spaces is called \textit{equivariant} or a $G$-\textit{map} if
$f(gx)=g(fx)$ for every $x\in X$ and $g\in G$. On the other hand, if $f(gx)=f(x)$ for every $x\in X$, we call $f$ an \textit{invariant} (or a $G$-\textit{invariant} map).

 A $G$-space $X$ is called \textit{proper} (in the sense of Palais \cite{Palais}) if it has an open cover consisting of, so called,  small sets.
 A set  $S\subset X$ is called \textit{small} if  any  point $x\in X$ has  a neighborhood $V$  such that the set $\langle S, V\rangle=\{g\in G\mid gS\cap V\neq\emptyset\}$  has compact closure in $G$.

In the following theorem we sumarize some useful properties about proper actions. The proof of these and other related results can be consulted in \cite{Palais}

\begin{theorem}\label{t:properties} Let $X$ be a Tychonoff space. If $G$ is a locally compact group acting properly on $X$ then the following conditions hold:
\begin{enumerate}[\rm(1)]
\item For every $x\in X$, the orbit $G(x)$ is closed, the stabilizer $G_x$ is compact and $G(x)$ is $G$-homeomorphic to $G/G_x$.
\item The orbit space $X/G$ is Tychonoff.
%\item If $X$ is metric and separable, then $X$ admits an invariant metric. Namely, the topology in $X$ can be generated by a metric $\rho:X\times X\to [0,\infty)$ such that 
%$$\rho(x, y)=\rho(gx, gy)\quad \text{for every }g\in G, \; x,y\in X.$$
\end{enumerate}
If additionally $G$ is a Lie group, then we also have that
\begin{enumerate}[\rm(3)]
\item Every orbit  $G(x)$ of $X$  is an equivariant retract of an invariant neighborhood of $G(x)$. Namely, for every $x\in X$ there exist an invariant neighborhood $U$ of $G(x)$ and a $G$-equivariant retraction $r:U\to G(x)$ . 

\end{enumerate}
\end{theorem}

Another well-known result that we will use is the following easy lemma. We include its proof for the sake of completeness.
\begin{lemma}\label{l:funcion de tychonoff} Let $X$ be a $G$-space, $x\in X$ an arbitrary point and $V\subset X$ a $G$-invariant neighborhood of the orbit $G(x)$.
\begin{enumerate}[\rm(1)]
\item If $X/G$ is regular, then we can find an invariant neighborhood $U$ of $G(x)$ such that
$$G(x)\subset U\subset\overline U\subset V.$$
\item If  $X/G$ is Tychonoff, then  we can find an invariant map 
$\lambda:X\to [0,1]$ such that $\lambda(y)=1$ and $\lambda (z)=0$ for every $y\in G(x)$ and $z\in X\setminus V$.
\end{enumerate}
\end{lemma}
\begin{proof}
(1)  Since the orbit map $\pi:X\to X/G$ is open, $\pi (V)$ is an open neighborhood of $\pi(x)$. If $X/G$ is regular, then we can find $ U_1\subset X/G$ an open neighborhood of $\pi(x)$ such that
$$\pi(x)\subset U_1\subset \overline{U_1}\subset \pi(V).$$
Thus 
$$ \pi^{-1}(\pi(x))=G(x)\subset \pi^{-1}(U_1)\subset \pi^{-1}(\overline {U_1})\subset \pi^{-1}(\pi(V))=V.$$
Let us call $U:= \pi^{-1}(U_1)$. Obviously $U$ is an open neighborhood of $G(x)$ and $\overline U\subset V$. Therefore $U$ is the desired neighborhood. 

(2). Now let us assume that $X/G$ is Tychonoff. Since $\pi(V)$ is an open neighborhood of $\pi(x)$ we can find a map $\widetilde \lambda:X/G\to[0,1]$ such that $\widetilde\lambda (\pi(x))=1$ and $\widetilde \lambda(\pi(z))=0$ for every $\pi(z)\in (X/G)\setminus  \pi(V)$. Now the desired map $\lambda:X\to [0,1]$ can be obviously defined by $\lambda=\widetilde\lambda\circ \pi. $
\end{proof}

Our interest in proper actions lies in the fact that the action of  $\Aff(n)$ on $\mathcal K_n$ is proper (\cite[Theorem 3.3]{NataliaFM}). Furthermore, since $\Aff(n)$ is a locally compact Lie group,  $\mathcal K_n$ satisfies all the implications of Theorem~\ref{t:properties}.  In particular the orbit space $\mathcal K_n/\Aff(n)$ is Tychonoff (in fact it is a compact metric space known as the non-symmetric Banach Mazur compactum) and then $\mathcal K_n$ satisfies all conditions of Lemma~\ref{l:funcion de tychonoff}.

\section{Gr\"unbaum's conjecture}

In order to simplify the notation, in what follows $G$ will always denotes the group $\Aff(n)$. 
Let $K\in\mathcal K_n$ be a convex body. Observe that in this case the set $\mathfrak{F}_n(K)$  coincides with $(\mathbb R^n)^{G_K}$, the set of all $G_K$-fixed points of $\mathbb R^n$, where $G_K$ denotes the stabilizer of $K$. Namely, a point $x\in\mathbb R^n$ belongs to $\mathfrak{F}_n(K)$  if and only if $G_K\subset G_x$.
On the other hand, the set $\mathfrak{P}_n$ is just the set of all $G$-equivariant  maps $p:\mathcal K_n\to\mathbb R^n$. 

If we translate Gr\"unbaum's conjecture to the language of topological transformation groups, we obtain  the following scenario: given a convex body $K\in \mathcal K_n$ and $x$ with $G_K\subset G_x$ we need to find a $G$-equivariant map $p:\mathcal K_n\to \mathbb R^n$ such that $p(K)=x$. 
We will construct such a map in the following theorem.

\begin{theorem}\label{t:main}
Let $K\in\mathcal K_n$. For every $x\in\mathfrak{F}_n(K) $  we can find an affine invariant point $p\in \mathfrak{P}_n$ such that 
$p(K)=x$.
\end{theorem}

\begin{proof}

Let us define $f:G(K)\to G(x)\subset \mathbb R^n$ as
\begin{equation}\label{e:funcion f}
f(gK)=gx.
\end{equation}

Since $x\in\mathfrak{F}_n(K) $  we have that  $G_K\subset G_x$ and therefore the map $f$ is well defined and continuous. By Theorem \ref{t:properties}, there exist an open invariant neighborhood $U$ of $G(K)$ and a $G$-equivariant retraction  $r:U\to G(K)$. 
By Lemma~\ref{l:funcion de tychonoff}, we can find  $V\subset \mathcal K_n$ another invariant neighborhood of $G(K)$ such that
$$G(K)\subset V\subset \overline V\subset U.$$
Using Lemma~\ref{l:funcion de tychonoff} again, we can also pick a  $G$-invariant map $\lambda:\mathcal K_n\to [0,1]$ such that $\lambda (gK)=1$ and $\lambda (A)=0$ for every $g\in G$ and $A\in \mathcal{K}_n\setminus V$.
Now let us consider $\varphi:\mathcal K_n\to \mathbb R^n$ any affine invariant point (for example, the centroid).
Finally we define $p:\mathcal K_n\to \mathbb R^n$ by the following rule:
\begin{equation}
p(A)=\begin{cases}
\varphi(A),&\text{if }A\in\mathcal K_n\setminus \overline{V},\\
\lambda(A)f( r(A))+(1-\lambda (A))\varphi(A),&\text{if }A\in U.
\end{cases}
\end{equation}
If $A\in U\cap \big( \mathcal K_n\setminus\overline{ V}\big)$, then $\lambda(A)=0$ and therefore 
$$p(A)=\lambda(A)f( r(A))+(1-\lambda (A))\varphi(A)=\varphi(A).$$ In the first place, $p$  is well defined and continuous (since it is continuous in each of the open sets $\mathcal K_n\setminus \overline{V}$ and $U$ where it is defined). Secondly, it is obvious that $p(K)=x$, so it only rests to show that $p$ is $G$-equivariant. Indeed, if $A\in\mathcal K_n\setminus \overline{V}$, we can use the fact that $\varphi$ is $G$-equivariant and get that 
$p(gA)=\varphi(gA)=g\varphi(A)=gp(A)$ for every $g\in G$.
On the other hand, if $A\in U$ and $g\in G$, then we can use the fact that $\lambda$ is $G$-invariant, $r$ and $f$ are $G$-equivariant and $g$ is an affine map, to infer that
\begin{align*}
p(gA)=&\lambda(gA)f( r(gA))+(1-\lambda (gA))\varphi(gA)\\
%&=\lambda(A)f\big( gr(A)\big)+(1-\lambda (A))g\varphi(A)\\
&=\lambda(A)gf( r(A))+(1-\lambda (A))g\varphi(A)\\
&=g\big (\lambda(A)f( r(A))+(1-\lambda (A))\varphi(A)\big).
\end{align*}
This equality completes the proof.
\end{proof}

If we are interested in proper affine invariant points, we can slightly modify the previous proof to get the following

\begin{theorem}
Let $K\in\mathcal K_n$ and $x\in \mathfrak{F}_n(K)$. If $x\in \Int(K)$, then we can find a proper affine invariant point $p\in \mathfrak{P}_n$ such that 
$p(K)=x$.
\end{theorem}

\begin{proof}
Let $f:G(K)\to G(x)$ be defined as in (\ref{e:funcion f}). 
Since $x\in \Int (K)$, for every $g\in G$ we have that,
\begin{equation*}
gx\in g\Int(K)=\Int(gK).
\end{equation*}
Using  Theorem~\ref{t:properties}, we can find a $G$-invariant neighborhood $\widetilde U$ of $G(K)$ and an equivariant retraction $\widetilde r:\widetilde U\to G(K)$.
%Let us denote $\widetilde r^{-1}(K)$ by $S$. Observe that $\widetilde U=G(S)$.

%Since $G$ is a Lie group, by Theorem~\ref{t:properties} there exists a compatible $G$-invariant metric $\rho$ in $\mathcal K^n$.
 By \cite[Lemma 3.1]{NataliaFM} we can  find $\varepsilon>0$  such that $B(x,\varepsilon)\subset \Int (A)$ for every $A\in O(K, \varepsilon)$ and $O(K,\varepsilon)\subset \widetilde U$. Since $f\circ \widetilde{r}$ is continuous we can also find $0<\delta\leq\varepsilon$ such that if $A\in O( K, \delta)$, then 
 \begin{equation}\label{e:contencion interior}
 f(\widetilde r(A))\in B(x, \varepsilon)\subset \Int (A).
 \end{equation}

 Now let us define $$U:=G\big(O(K,\delta )\big)=\{gA\mid A\in O(K,\delta)\}.$$ 
% 
% 
%$$U:=G\big(S\cap B_\rho(K,\varepsilon) \big)=G(S)\cap G( B_\rho(K,\varepsilon))$$
Clearly $U$ is an open invariant neighborhood of $G(K)$ such that $U\subset \widetilde U$.
Let $r$ be the restriction of $\widetilde r$ to $U$. 
Observe that for every $A\in U$, there exists $g\in G$ and $A'\in O(K,\delta )$ with $gA'=A$. Therefore, by (\ref{e:contencion interior}), 
\begin{equation*}
f(r(A))=f(\widetilde r(gA'))=gf(\widetilde r(A'))\in g\Int (A')=\Int (gA')=\Int(A).
\end{equation*}

Now we continue the proof verbatim as in the proof of Theorem~\ref{t:main}, taking as the map $\varphi $ any proper affine invariant point (for example, the centroid). 
To finish the proof we need to prove that the obtained map $p$ is proper.
Indeed, if $A\in\mathcal K_n\setminus \overline V$ then  $p(A)=\varphi(A)\in \Int (A)$, because $\varphi$ is proper.
On the other hand, if $A\in U$, 
%$A=gA'$ for some $A'\in  B_{d_H}(K,\delta)$. In this case $r(A)=r(gA')=gr(A')=gK$. Thus
%$$f(r(A))=f(gK)=gx.$$
%In addition, since $A'\in B_rho(K, \varepsilon)$, then $x\in \Int A'$ and therefore
%$$f(r(A))=gx\in g\Int (A')=\Int(gA')=\Int(A).$$
 observe that 
\begin{align*}
p(A)=\lambda (A)f(r(A))+(1-\lambda(A))\varphi(A)
\end{align*}
is a convex combination of the points $f(r(A))$ and $\varphi(A)$ that lie in the convex set $\Int(A)$. Thus $p(A)\in \Int(A)$, as desired.

\end{proof}

\end{document}